\documentclass[a4paper]{cas-sc}

\usepackage[numbers,sort&compress]{natbib}
\usepackage{amsmath,amssymb,amsfonts,amsthm}
\usepackage{graphicx}
\usepackage{xcolor}
\usepackage{microtype}

\makeatletter
\@fleqnfalse
\makeatother

\AtBeginDocument{%
  \setlength{\abovedisplayskip}{6pt plus 2pt minus 2pt}%
  \setlength{\belowdisplayskip}{6pt plus 2pt minus 2pt}%
  \setlength{\abovedisplayshortskip}{3pt plus 2pt minus 1pt}%
  \setlength{\belowdisplayshortskip}{4pt plus 2pt minus 2pt}%
  \setlength{\textfloatsep}{10pt plus 2pt minus 2pt}%
  \setlength{\floatsep}{8pt plus 2pt minus 2pt}%
  \setlength{\intextsep}{8pt plus 2pt minus 2pt}%
  \setlength{\jot}{2pt}%
}

\ExplSyntaxOn
\cs_set:Npn \__first_footerline:
{
  \group_begin:
  \small\sffamily
  \ifnum\theblind>0\relax
  \else
    \__short_authors:
  \fi
  \group_end:
}
\ExplSyntaxOff

\DeclareMathOperator{\orb}{orb}
\newcommand{\zeroone}{\mbox{$0$-$1$}}

\theoremstyle{plain}
\newtheorem{theorem}{Theorem}
\newtheorem{lemma}{Lemma}[section]
\newtheorem{proposition}{Proposition}[section]
\newtheorem{corollary}{Corollary}[section]

\theoremstyle{definition}
\newtheorem{definition}{Definition}

\theoremstyle{remark}
\newtheorem*{remark}{Remark}
\newtheorem*{problem}{Problem}

\begin{document}
\let\WriteBookmarks\relax

\shorttitle{Minimal Balanced Collections}
\shortauthors{M. V. Bludov and N. K. Zuev}

\title[mode=title]{Enumerating Minimal Balanced Collections}

\author[1]{Mikhail V. Bludov}
\cormark[1]
\ead{bludov.mv@phystech.edu}

\author[1]{Nikolai K. Zuev}
\ead{zuev.nk@phystech.edu}

\affiliation[1]{organization={Laboratory of Combinatorial and Geometric Structures, Moscow Institute of Physics and Technology},
  addressline={9 Institutskiy per.},
  city={Dolgoprudny},
  postcode={141701},
  state={Moscow Region},
  country={Russian Federation}}

\cortext[1]{Corresponding author}

\begin{abstract}
In this note, we explore the combinatorics of balanced collections. A collection of subsets of the set $[n] = \{1, \dots, n\}$ is called \emph{balanced} if the relative interior of the convex hull of the corresponding characteristic vectors intersects the main diagonal of the $n$-dimensional cube at a point other than the origin, and it is called \emph{minimal} if it contains no proper balanced subcollections. We determine the asymptotic number of minimal balanced collections. Specifically, if $B_n$ denotes their total number, then $B_n=2^{n^2-n+1}\bigl(1+o(1)\bigr)/n!$ as $n\to\infty$. We discuss applications of this result to fractional matchings in hypergraphs and to resonance arrangements.
\end{abstract}

\begin{keywords}
balanced collection \sep $0$--$1$ matrix \sep fractional matching \sep column complementation \sep asymptotic enumeration
\MSC[2020]{05A16, 15B34}
\end{keywords}

\maketitle

\section{Introduction}

This paper studies the combinatorics and enumeration of balanced
collections, a class of set systems admitting positive weights that cover
each element of the ground set equally.

Balanced collections were introduced by Bondareva~\cite{Bon} and
Shapley~\cite{Sh67} in their characterization of cooperative games with
nonempty cores.

\begin{definition}
Let $\Phi\subseteq 2^{[n]}$, and let
$\mathbf{1}=(1,\ldots,1)^{\mathsf T}\in\mathbb{R}^n$.  We call $\Phi$
\emph{balanced} if there exist positive weights
$(\lambda_S)_{S\in\Phi}$ such that
$\sum_{S\in\Phi}\lambda_S\mathbf{1}_S=\mathbf{1}$, where
$\mathbf{1}_S$ is the characteristic vector of $S$.  A balanced
collection is \emph{minimal} if none of its proper subcollections is
balanced.
\end{definition}

This concept admits a natural geometric interpretation.  Each subset
$S\subseteq[n]$ corresponds to a vertex of the cube $[0,1]^n$.  A collection
$\Phi$ is balanced if and only if the relative interior of the convex hull of
the corresponding vertices intersects the main diagonal of the cube at a
point other than $\mathbf{0}$.

Balanced collections are closely related to cooperative game theory.  A
cooperative game with transferable utility has a nonempty core if and only
if a certain system of inequalities holds for all minimal balanced
collections.  In~\cite{BalGame}, the authors study the geometric structure of
the set of balanced games and show that this set forms a cone (or, under
additional constraints, a polytope).  They further prove that the facets of
this cone are in bijection with the nontrivial minimal balanced collections,
that is, all minimal balanced collections except $\{[n]\}$. 

Minimal balanced collections also generalize set partitions: every partition
of $[n]$ is a minimal balanced collection. Balanced collections admit several further formulations and interpretations. For example, their relation to fractional perfect matchings is discussed in Appendix~\ref{app:fractional-matchings}, and their connection with the resonance arrangement is discussed in Appendix~\ref{app:resonance-circuits}. Thus the problem
of enumerating minimal balanced collections arises naturally.

Despite their simple definition, minimal balanced collections are difficult
to enumerate.  Peleg developed an inductive procedure that constructs the
minimal balanced collections on $[n+1]$ from those on $[n]$~\cite{Peleg}.
Laplace Mermoud, Grabisch, and Sudh\"olter later implemented this procedure
and used it to generate all minimal balanced collections, and hence determine
their number, for $n\leq7$~\cite{Mermoud}.  Some restricted cases are better
understood; for example, minimal balanced collections consisting entirely of
two-element sets were characterized in~\cite{BluMus}.

For comparison, Billera et al.~\cite{Billera2} call a collection
\emph{unbalanced} if the convex hull of its characteristic vectors does not
meet the main diagonal of $[0,1]^n$, and \emph{maximal unbalanced} if it is
inclusion-maximal with this property.  If $E_n$ denotes the number of maximal
unbalanced collections on $[n]$, they proved that, for $n\geq3$,
\[
2^{(n-1)(n-2)/2}<E_n<2^{(n-1)^2}.
\]

Let $B_n$ denote the number of minimal balanced collections on $[n]$.  Our
main result determines its asymptotic behavior.

\begin{theorem}\label{MainTh}
As $n\to\infty$, $B_n=\frac{2^{n^2-n+1}}{n!}\bigl(1+o(1)\bigr)$.
\end{theorem}

The first arXiv version of this paper, posted on 24 November 2025, established
explicit upper and lower bounds for $B_n$ that differed by a factor of order
$2^n$.  It already contained the matrix correspondence now presented in
Section~2 and the analysis of weight vectors in Section~3, and explored
column-complementation symmetry in Section~4.  The initial argument and these
bounds are now presented in Appendix~B.

In January 2026, Garc\'ia-Segador et al.~\cite{ClosestBalancedGame} also
obtained the asymptotic formula in Theorem~\ref{MainTh}.  Their paper centers
on the closest balanced game and introduces the least square core, a solution
concept for cooperative games.  They use this enumeration result to analyze
the projection of a random game onto the set of balanced games.  Their
enumeration proof combines an orthant-symmetry argument with Tikhomirov's
singularity theorem.

The present asymptotic proof was obtained independently of
\cite{ClosestBalancedGame}.  It closes the gap between the original bounds by
combining the earlier results with Tikhomirov's singularity theorem.
Tikhomirov's theorem is the only probabilistic input in
Section~5; the remaining enumeration argument is deterministic.  Additionally,
in Appendices~C and~D, we discuss applications of our results to fractional
matchings in hypergraphs and to resonance arrangements.

\subsection*{Structure of the paper}

Section~2 introduces the matrix notation and establishes the correspondence
between minimal balanced collections and \zeroone\ matrices.  Section~3
explores properties of weight
vectors.  Section~4 studies the column-complementation action.  Section~5 proves
Theorem~\ref{MainTh}.  Appendix~A gives an exact formula for $B_{n,m}$ when
$m$ is fixed, and Appendix~B gives explicit upper and lower bounds on $B_n$
for finite $n$ without using Tikhomirov's theorem.  Appendix~C discusses
fractional perfect matchings, and Appendix~D records the consequences for
circuits of the resonance arrangement.

\section{Notation and preliminaries}\label{sec:preliminaries}

Throughout the paper, $1\leq m\leq n$.  For every positive integer $d$, write
$[d]=\{1,\ldots,d\}$ and let
$\mathbf{1}_d=(1,\ldots,1)^{\mathsf T}\in\mathbb{R}^d$.  We abbreviate
$\mathbf{1}_n$ to $\mathbf{1}$.  Let $\mathcal{M}_{n,m}$ denote the set of
all $n\times m$ matrices with entries in $\{0,1\}$, viewed as real matrices.
We identify each subset $S\subseteq[n]$ with its characteristic vector
$\mathbf{1}_S$.  Similarly, we identify a \zeroone\ row of length $m$ with a
subset of $[m]$.

Every $M\in\mathcal{M}_{n,m}$ can be written as
$M=(\mathbf{1}_{S_1},\ldots,\mathbf{1}_{S_m})$.  If the columns are distinct,
we associate with $M$ the collection
$\Phi(M)=\{S_1,\ldots,S_m\}$.
Conversely, any ordering of a collection $\Phi=\{S_1,\ldots,S_m\}$ determines
a matrix with columns $\mathbf{1}_{S_1},\ldots,\mathbf{1}_{S_m}$.  Thus
$\Phi(M)$ records the set of columns but not their order.

For a matrix $M\in\mathcal{M}_{n,m}$ of rank $m$, the equation
$M\lambda=\mathbf{1}$ has at most one solution.  Let
\[
\mathcal{I}_{n,m}
:=\{M\in\mathcal{M}_{n,m}:\operatorname{rank}M=m
\text{ and }M\lambda=\mathbf{1}
\text{ for some }\lambda\in\mathbb{R}^m\}.
\]
For $M\in\mathcal{I}_{n,m}$, the unique solution is denoted by $\lambda(M)$
and called the \emph{weight vector} of $M$.  Its coordinates are rational:
apply Cramer's rule to any nonsingular $m\times m$ row submatrix of $M$.
We write
\[
\begin{aligned}
\mathcal{I}_{n,m}^{\neq0}
&:=\{M\in\mathcal{I}_{n,m}:\lambda(M)_j\neq0\text{ for every }j\in[m]\},
&
\mathcal{I}_{n,m}^{>0}
&:=\{M\in\mathcal{I}_{n,m}:\lambda(M)_j>0\text{ for every }j\in[m]\}.
\end{aligned}
\]

Let $B_{n,m}$ denote the number of minimal balanced collections
consisting of $m$ subsets, and let $B_n$ denote their total number.

The following proposition gives the basic correspondence between minimal
balanced collections and the matrix classes defined above.

\begin{proposition}\label{prop:matrix-correspondence}
If $M\in\mathcal{I}_{n,m}^{>0}$, then its columns are distinct and
$\Phi(M)$ is a minimal balanced collection.  Conversely, if
$\Phi=\{S_1,\ldots,S_m\}$ is a minimal balanced collection, then every matrix
whose columns are $\mathbf{1}_{S_1},\ldots,\mathbf{1}_{S_m}$, in any order,
belongs to $\mathcal{I}_{n,m}^{>0}$.
\end{proposition}

\begin{proof}
Suppose first that $M\in\mathcal{I}_{n,m}^{>0}$.  Since $M$ has full column
rank, its columns are distinct.  Its positive weight vector $\lambda(M)$
shows that $\Phi(M)$ is balanced.  If a proper subcollection were balanced,
we could extend its balancing weights by zeros to obtain a vector
$\mu\in\mathbb{R}^m$ satisfying $M\mu=\mathbf{1}$.  This contradicts the
uniqueness of $\lambda(M)$, because $\lambda(M)$ has no zero coordinates.
Hence $\Phi(M)$ is a minimal balanced collection.

Conversely, let $\Phi=\{S_1,\ldots,S_m\}$ be minimal balanced, and let $M$ be
a matrix whose columns are the corresponding characteristic vectors.  There
is a vector $\lambda\in\mathbb{R}_{>0}^m$ such that
$M\lambda=\mathbf{1}$.  Suppose that $\operatorname{rank}M<m$, and choose a
nonzero vector $v\in\ker M$.  Replacing $v$ by $-v$ if necessary, assume that
$v_i<0$ for some $i$, and set
\[
t:=\min_{i:v_i<0}\frac{\lambda_i}{-v_i},
\qquad
\mu:=\lambda+tv.
\]
Then $\mu\geq0$, at least one coordinate of $\mu$ is zero, and
$M\mu=\mathbf{1}$.  The sets indexed by
$\{i\in[m]:\mu_i>0\}$ therefore form a proper balanced subcollection of
$\Phi$, a contradiction.  Thus $\operatorname{rank}M=m$.  It follows that
$M\in\mathcal{I}_{n,m}$ and that its unique weight vector is the positive
vector $\lambda$; hence $M\in\mathcal{I}_{n,m}^{>0}$.
\end{proof}

In particular, every minimal balanced collection on $[n]$ has at most $n$
members and has a unique system of rational balancing weights.

\begin{lemma}\label{L1}
For all $1\leq m\leq n$, $B_{n,m}=\frac{|\mathcal{I}_{n,m}^{>0}|}{m!}$.
\end{lemma}

\begin{proof}
By Proposition~\ref{prop:matrix-correspondence}, the matrices in
$\mathcal{I}_{n,m}^{>0}$ are precisely the ordered versions of minimal
balanced collections with $m$ members.  Since their columns are distinct,
each such collection has exactly $m!$ orderings.
\end{proof}

\section{Weight vectors}

For $\lambda\in\mathbb{R}^m$, let
$U(\lambda):=\{u\in\{0,1\}^m:u\lambda=1\}$.
We call the elements of $U(\lambda)$ the \emph{unifying rows} of $\lambda$.
For any set $V$ of vectors, write $\operatorname{rank}V$ for the dimension of
its linear span.  If $A\in\mathcal{I}_{n,m}$ has weight vector $\lambda$, then
every row of $A$ belongs to $U(\lambda)$.

\begin{lemma}
Let $\lambda\in\mathbb{R}^m$ and $n\geq m$.  There exists a matrix
$A\in\mathcal{I}_{n,m}$ with $\lambda(A)=\lambda$ if and only if
$\operatorname{rank}U(\lambda)=m$.
If the coordinates of $\lambda$ are all nonzero, respectively all positive,
then such a matrix belongs to $\mathcal{I}_{n,m}^{\neq0}$, respectively
$\mathcal{I}_{n,m}^{>0}$.
\end{lemma}

\begin{proof}
If $A\in\mathcal{I}_{n,m}$ has weight vector $\lambda$, then the rows of $A$
belong to $U(\lambda)$.  Hence
\[
m=\operatorname{rank}A\leq\operatorname{rank}U(\lambda)\leq m.
\]
Conversely, suppose that $\operatorname{rank}U(\lambda)=m$.  Choose linearly
independent rows $u_1,\ldots,u_m\in U(\lambda)$ and, if $n>m$, repeat any of
them to obtain $n$ rows.  The resulting matrix $A$ has rank $m$ and satisfies
$A\lambda=\mathbf{1}$, so $A\in\mathcal{I}_{n,m}$ and
$\lambda(A)=\lambda$.  The final assertions follow directly from the
definitions of $\mathcal{I}_{n,m}^{\neq0}$ and
$\mathcal{I}_{n,m}^{>0}$.
\end{proof}

It follows that the set of positive weight vectors is independent of
$n\geq m$.  We denote it by
\[
\Lambda_m
:=\{\lambda\in\mathbb{R}_{>0}^m:
      \operatorname{rank}U(\lambda)=m\}.
\]
Equivalently, $\Lambda_m=\{\lambda(A):A\in\mathcal{I}_{n,m}^{>0}\}$ for
every $n\geq m$.

For each integer $k$ with $1\leq k<m$, the vector
$k^{-1}\mathbf{1}_m$ belongs to $\Lambda_m$.  Indeed,
$U(k^{-1}\mathbf{1}_m)$ is the family of all $k$-element subsets of $[m]$,
whose characteristic vectors span $\mathbb{R}^m$.

\begin{lemma}\label{minU}
For every $m\geq1$, $\min_{\lambda\in\Lambda_m}|U(\lambda)|=m$.
\end{lemma}

\begin{proof}
Every $\lambda\in\Lambda_m$ satisfies
$\operatorname{rank}U(\lambda)=m$, so $|U(\lambda)|\geq m$.  Equality is
attained at $\lambda=\mathbf{1}_m$, for which $U(\lambda)$ consists of the $m$
standard basis vectors.
\end{proof}

\begin{lemma}\label{maxU}
For every $m\geq1$,
$\max_{\lambda\in\Lambda_m}|U(\lambda)|=\binom{m}{\lceil m/2\rceil}$.
\end{lemma}

\begin{proof}
Let $r=\lceil m/2\rceil$.  For $m\geq2$, the vector
$\lambda=r^{-1}\mathbf{1}_m$ belongs to $\Lambda_m$, and $U(\lambda)$ is the
family of all $r$-element subsets of $[m]$.  It therefore has cardinality
$\binom{m}{r}$.  The same conclusion is immediate when $m=1$.

For the upper bound, fix $\lambda\in\Lambda_m$.  Since every coordinate of
$\lambda$ is positive, no two distinct members of $U(\lambda)$ are
comparable by inclusion: if $v\subsetneq u$, then $(u-v)\lambda>0$,
whereas $u\lambda=v\lambda=1$.  Thus $U(\lambda)$ is an antichain in the
Boolean lattice.  Sperner's theorem gives
$|U(\lambda)|\leq\binom{m}{\lceil m/2\rceil}$.
\end{proof}

\section{Column complementation and orbit bounds}

For $I\subseteq[m]$ and $M=(a_{*1},\ldots,a_{*m})\in\mathcal{M}_{n,m}$,
let $z_I(M)$ be the matrix obtained by replacing each column $a_{*j}$ with
$\mathbf{1}-a_{*j}$ for $j\in I$.  Identifying subsets of $[m]$ with their
characteristic vectors in $\mathbb{F}_2^m$, the maps $z_I$ define a free
action of the additive group $\mathbb{F}_2^m$ on $\mathcal{M}_{n,m}$, where
$z_\varnothing$ is the identity and
$z_I\circ z_J=z_{I\mathbin{\vartriangle}J}$, where $\vartriangle$ denotes
symmetric difference.  We write $\orb(M)$ for
the orbit of $M$ and set
\[
\orb_{\neq0}(M):=\orb(M)\cap\mathcal{I}_{n,m}^{\neq0},
\qquad
\orb_{>0}(M):=\orb(M)\cap\mathcal{I}_{n,m}^{>0}.
\]

\begin{lemma}\label{Iaction}
Let $M\in\mathcal{I}_{n,m}^{\neq0}$ have weight vector $\lambda$, let
$I\subseteq[m]$, and put $s_I=\sum_{j\in I}\lambda_j$.  If $s_I=1$, then
$\operatorname{rank}z_I(M)=m-1$.
If $s_I\neq1$, then $z_I(M)\in\mathcal{I}_{n,m}^{\neq0}$ and
\begin{equation}\label{eq:inverted-weights}
\lambda(z_I(M))_j
=\begin{cases}
\displaystyle-\frac{\lambda_j}{1-s_I},&j\in I,\\[1.2ex]
\displaystyle \frac{\lambda_j}{1-s_I},&j\notin I.
\end{cases}
\end{equation}
Consequently, if $s_I<1$, the signs change exactly on $I$; if $s_I>1$,
they change exactly on $[m]\setminus I$.
\end{lemma}

\begin{proof}
Write $M=(a_{*1},\ldots,a_{*m})$.  To determine the rank of $z_I(M)$,
consider a linear relation among its columns.  Thus, let
$\alpha=(\alpha_1,\ldots,\alpha_m)^{\mathsf T}\in\mathbb{R}^m$ be the vector
of coefficients in such a relation:
\[
\sum_{j\notin I}\alpha_j a_{*j}
+\sum_{j\in I}\alpha_j(\mathbf{1}-a_{*j})=0.
\]
Put $c=\sum_{j\in I}\alpha_j$.  Since
$\mathbf{1}=\sum_{j=1}^m\lambda_j a_{*j}$, this relation becomes
\[
\sum_{j\notin I}(\alpha_j+c\lambda_j)a_{*j}
+\sum_{j\in I}(-\alpha_j+c\lambda_j)a_{*j}=0.
\]
The columns of $M$ are independent.  Therefore
\[
\alpha_j=
\begin{cases}
c\lambda_j,&j\in I,\\
-c\lambda_j,&j\notin I.
\end{cases}
\]
Summing over $j\in I$ gives $c=cs_I$.  If $s_I\neq1$, then $c=0$ and
$\alpha=0$, so the columns of $z_I(M)$ are independent.  If $s_I=1$, the
displayed coefficients with $c=1$ give a nonzero relation, and every relation
is a scalar multiple of it.  Hence $\operatorname{rank}z_I(M)=m-1$.

Suppose now that $s_I\neq1$.  The product of $z_I(M)$ with the vector on the
right-hand side of~\eqref{eq:inverted-weights} is
\[
\frac{1}{1-s_I}
\left(
\sum_{j\notin I}a_{*j}\lambda_j
-\sum_{j\in I}(\mathbf{1}-a_{*j})\lambda_j
\right)
=\frac{M\lambda-s_I\mathbf{1}}{1-s_I}
=\mathbf{1}.
\]
Thus $z_I(M)\in\mathcal{I}_{n,m}$, and the displayed vector is its unique
weight vector.  None of its coordinates vanishes, and the sign assertions
follow.
\end{proof}

\begin{corollary}\label{OrbandUnif}
For every $M\in\mathcal{I}_{n,m}^{\neq0}$,
$|\orb_{\neq0}(M)|=2^m-|U(\lambda(M))|$.
Moreover, $|U(\lambda(M'))|=|U(\lambda(M))|$ for every
$M'\in\orb_{\neq0}(M)$.
\end{corollary}

\begin{proof}
Because the action is free, the matrices in $\orb(M)$ are in
bijection with the subsets $I\subseteq[m]$.  By Lemma~\ref{Iaction}, the
matrix $z_I(M)$ fails to belong to $\mathcal{I}_{n,m}^{\neq0}$ precisely when
\(
\sum_{j\in I}\lambda(M)_j=1.
\)
The number of such subsets is $|U(\lambda(M))|$, which proves the formula.
If $M'\in\orb_{\neq0}(M)$, then $M$ and $M'$ have the same orbit, so
the formula also gives $|U(\lambda(M'))|=|U(\lambda(M))|$.
\end{proof}

\begin{lemma}
\label{PosOrb}
For $m\geq2$ and every $M\in\mathcal{I}_{n,m}^{\neq0}$,
$|\orb_{>0}(M)|=2$.
\end{lemma}

\begin{proof}
Let $\lambda=\lambda(M)$ and set
$I_+=\{j\in[m]:\lambda_j>0\}$ and
$I_-=\{j\in[m]:\lambda_j<0\}$.
Put $p=\sum_{j\in I_+}\lambda_j$.  In each row equation of
$M\lambda=\mathbf{1}$, the contribution from indices in $I_+$ is at most
$p$, whereas the contribution from indices in $I_-$ is nonpositive.  Thus
$1\leq p$.
If $p=1$, every row of $M$ has entry $1$ in the columns indexed by $I_+$ and
entry $0$ in those indexed by $I_-$.  Thus the former columns are
$\mathbf{1}$ and the latter are $\mathbf{0}$, giving
$\operatorname{rank}M\leq1$, contrary to $m\geq2$.  Hence $p>1$, whereas
$\sum_{j\in I_-}\lambda_j\leq0<1$.

Lemma~\ref{Iaction} now shows that both $z_{I_+}(M)$ and $z_{I_-}(M)$ have
positive weight vectors.  Conversely, the sign rule in that lemma shows that
these are the only matrices in the orbit with positive weight vectors.  They
are distinct because the action is free and $I_+\neq I_-$.
\end{proof}

\begin{lemma}

\label{OrbLemma}
For $m\geq2$,
\[
\frac{2}{2^m-m}\,|\mathcal{I}_{n,m}^{\neq0}|
\;\le\;
|\mathcal{I}_{n,m}^{>0}|
\;\le\;
\frac{2}{2^m-\binom{m}{\lceil m/2\rceil}}\,|\mathcal{I}_{n,m}^{\neq0}|.
\]
\end{lemma}

\begin{proof}
Let $\mathcal{O}$ be an orbit meeting
$\mathcal{I}_{n,m}^{\neq0}$.  By Lemma~\ref{PosOrb}, it contains a matrix
$M_+\in\mathcal{I}_{n,m}^{>0}$ and
$|\mathcal{O}\cap\mathcal{I}_{n,m}^{>0}|=2$.
Since $\lambda(M_+)\in\Lambda_m$, Corollary~\ref{OrbandUnif} and
Lemmas~\ref{minU}--\ref{maxU} give
\[
|\mathcal{O}\cap\mathcal{I}_{n,m}^{\neq0}|
  =2^m-|U(\lambda(M_+))|,
\qquad
m\leq|U(\lambda(M_+))|
  \leq\binom{m}{\lceil m/2\rceil}.
\]
Therefore
\[
\frac{2}{2^m-m}
\leq
\frac{|\mathcal{O}\cap\mathcal{I}_{n,m}^{>0}|}
{|\mathcal{O}\cap\mathcal{I}_{n,m}^{\neq0}|}
\leq
\frac{2}{2^m-\binom{m}{\lceil m/2\rceil}}.
\]
Multiplying by $|\mathcal{O}\cap\mathcal{I}_{n,m}^{\neq0}|$ and summing over
all such orbits proves the result.
\end{proof}

\begin{remark}
Lemma~\ref{PosOrb} has a useful geometric interpretation.  Let $m\geq3$, and
let $v_1,\ldots,v_m$ be unit vectors in $\mathbb{R}^{m-1}$ such that every
$m-1$ of them are linearly independent, and consider the $m$ antipodal pairs
$\{v_i,-v_i\}$.  Among the $2^m$ ways of choosing one point from each pair,
exactly two give a simplex whose convex hull contains the origin.  The
existence of such a choice is also a special case of the colorful
Carath\'eodory theorem; see~\cite{Barany}.

For $m=4$, this observation is a key ingredient in Problem A--6 from the
53rd Putnam Mathematical Competition~\cite{Putnam}:
\begin{problem}[A--6]
Four points are chosen independently and uniformly at random on the surface
of a sphere.  What is the probability that the center of the sphere lies
inside the tetrahedron formed by these four points?
\end{problem}

A detailed solution of this problem, together with higher-dimensional
generalizations, appears in~\cite{GeneralPutnam}.
The video~\cite{TheHardestProblem} provides an accessible introduction to
this problem.
\end{remark}

\section{Asymptotic enumeration}

We now estimate the number of square matrices with nonzero weight vector.
The only probabilistic input is Tikhomirov's
Theorem~A~\cite{Tikhomirov2020}, restated below.

\begin{theorem}[Tikhomirov]\label{thm:tikhomirov}
Let $A_r$ be an $r\times r$ random matrix whose entries are independent and
uniform on $\{0,1\}$.  Then
\[
\mathbb{P}(\det A_r=0)
=\left(\frac12+o_r(1)\right)^r
\qquad\text{as }r\to\infty.
\]
\end{theorem}

For $r\geq1$, let
$\mathcal{S}_r:=\{A\in\mathcal{M}_{r,r}:\det A=0\}$ and
$q_r:=|\mathcal{S}_r|/2^{r^2}$.
Thus $q_r$ is the proportion of singular \zeroone\ matrices of order $r$.
Theorem~\ref{thm:tikhomirov} gives
$q_r=\left(\frac12+o_r(1)\right)^r$ as $r\to\infty$.

\begin{lemma}\label{lem:almost-all-nonzero}
As $n\to\infty$, $|\mathcal{I}_{n,n}^{\neq0}|=2^{n^2}\bigl(1-o(1)\bigr)$.
\end{lemma}

\begin{proof}
For $M\in\mathcal{M}_{n,n}$ and $j\in[n]$, let $M^{(j)}$ be obtained by
replacing the $j$th column of $M$ by $\mathbf{1}$.  If $\det M\neq0$, then
$\lambda(M)=M^{-1}\mathbf{1}$, and Cramer's rule gives
$\lambda(M)_j=\det M^{(j)}/\det M$.
Hence $M\in\mathcal{I}_{n,n}^{\neq0}$ precisely when $\det M\neq0$ and
$\det M^{(j)}\neq0$ for every $j\in[n]$.

For $j\in[n]$, put
$\mathcal{E}_j:=\{M\in\mathcal{M}_{n,n}:\det M^{(j)}=0\}$.
Let $\widehat{\mathcal{S}}_n$ be the set of singular \zeroone\ matrices of
order $n$ whose last column is $\mathbf{1}$.  Clearly,
$|\mathcal{E}_j|=2^n|\widehat{\mathcal{S}}_n|$.  For $n\geq2$, we claim that
$|\widehat{\mathcal{S}}_n|=2^{n-1}|\mathcal{S}_{n-1}|$.  To prove this, fix
$x\in\{0,1\}^{n-1}$.  The matrices with last column
$\mathbf{1}$ and first row $(x^{\mathsf T},1)$ are
\[
A(x,Y):=
\begin{pmatrix}
x^{\mathsf T} & 1\\
Y & \mathbf{1}_{n-1}
\end{pmatrix},
\qquad
Y\in\mathcal{M}_{n-1,n-1}.
\]
If $y_k$ is the $k$th column of $Y$, let $\tau_x(Y)$ be the \zeroone\ matrix
whose $k$th column is
\[
\begin{cases}
y_k, & x_k=0,\\
\mathbf{1}_{n-1}-y_k, & x_k=1.
\end{cases}
\]
The map $\tau_x$ is an involution: applying it twice returns $Y$.

For every $k$ with $x_k=1$, subtract the last column of $A(x,Y)$ from
column $k$ and then multiply column $k$ by $-1$.  Leave the remaining
columns unchanged.  These elementary column operations preserve
singularity and transform $A(x,Y)$ into
\[
\begin{pmatrix}
0 & 1\\
\tau_x(Y) & \mathbf{1}_{n-1}
\end{pmatrix}.
\]
Expanding the determinant of the transformed matrix along its first row
shows that $A(x,Y)$ is singular if and only if $\tau_x(Y)$ is singular.
Since $\tau_x$ is a bijection, exactly $|\mathcal{S}_{n-1}|$ choices of
$Y$ make $A(x,Y)$ singular.  Summing over the $2^{n-1}$ choices of $x$
proves the claim.  Together with the preceding count of $\mathcal{E}_j$,
it gives $|\mathcal{E}_j|=2^{2n-1}|\mathcal{S}_{n-1}|=2^{n^2}q_{n-1}$.
By the Cramer-rule characterization above, every matrix outside
$\mathcal{I}_{n,n}^{\neq0}$ belongs either to $\mathcal{S}_n$ or to at least
one of the sets $\mathcal{E}_j$.  Hence
$2^{n^2}-|\mathcal{I}_{n,n}^{\neq0}|
\leq|\mathcal{S}_n|+\sum_{j=1}^n|\mathcal{E}_j|$.
After division by $2^{n^2}$ and application of
Theorem~\ref{thm:tikhomirov}, this becomes
\[
1-\frac{|\mathcal{I}_{n,n}^{\neq0}|}{2^{n^2}}
\leq q_n+nq_{n-1}
=\left(\frac12+o_n(1)\right)^n
+n\left(\frac12+o_{n-1}(1)\right)^{n-1}
=o(1).
\]
The result follows.
\end{proof}
\begin{corollary}\label{cor:Bnn-asymptotic}
As $n\to\infty$, $B_{n,n}=\frac{2^{n^2-n+1}}{n!}\bigl(1+o(1)\bigr)$.
\end{corollary}

\begin{proof}
For $n\geq2$, Lemmas~\ref{L1} and~\ref{OrbLemma} give
\[
\frac{2|\mathcal{I}_{n,n}^{\neq0}|}{n!(2^n-n)}
\leq B_{n,n}\leq
\frac{2|\mathcal{I}_{n,n}^{\neq0}|}
{n!\left(2^n-\binom{n}{\lceil n/2\rceil}\right)}.
\]
Apply Lemma~\ref{lem:almost-all-nonzero} and use
$n/2^n=o(1)$ and
$\binom{n}{\lceil n/2\rceil}/2^n=O(n^{-1/2})$.
Both denominators in the preceding bounds are therefore
$2^n(1-o(1))$, and the stated asymptotic formula follows.
\end{proof}

\begin{proof}[Proof of Theorem~\ref{MainTh}]
We first show that collections with fewer than $n$ members are negligible.
For $2\leq m<n$, Lemmas~\ref{L1} and~\ref{OrbLemma}, together with
$|\mathcal{I}_{n,m}^{\neq0}|\leq2^{nm}$ and
$\binom{m}{\lceil m/2\rceil}\leq2^{m-1}$, give
\[
B_{n,m}
\leq
\frac{2|\mathcal{I}_{n,m}^{\neq0}|}
{m!\left(2^m-\binom{m}{\lceil m/2\rceil}\right)}
\leq
\frac{2^{nm-m+2}}{m!}.
\]
Combining this estimate with Corollary~\ref{cor:Bnn-asymptotic}, we obtain
$B_{n,m}/B_{n,n}\leq
2\bigl(1+o(1)\bigr)(n!/m!)2^{(m-n)(n-1)}$.
Put $k=n-m$.  Since $n!/m!\leq n^k$,
\[
\sum_{m=2}^{n-1}\frac{B_{n,m}}{B_{n,n}}
\leq
2\bigl(1+o(1)\bigr)
\sum_{k=1}^{n-2}\left(n2^{-(n-1)}\right)^k
=O(n2^{-n}).
\]
Finally, $B_{n,1}=1$, so $B_{n,1}/B_{n,n}=o(1)$.  Therefore
$B_n=B_{n,n}(1+o(1))$, and the result follows from
Corollary~\ref{cor:Bnn-asymptotic}.  Explicit bounds valid for every finite
$n$ are given in Appendix~\ref{app:finite-bounds}.
\end{proof}

\appendix

\section{Enumeration of \texorpdfstring{$B_{n,m}$}{B(n,m)} for fixed
\texorpdfstring{$m$}{m}}

The weight-vector description also gives a constructive procedure for computing
$B_{n,m}$ when $m$ is fixed.  We record this procedure and some of its
consequences separately from the asymptotic argument above.

The set $\Lambda_m$ is finite.  Indeed, for every $\lambda\in\Lambda_m$ we
may choose $m$ linearly independent vectors from $U(\lambda)$ and use them
as the rows of an $m\times m$ matrix $A$ with entries in $\{0,1\}$.  Then
$A\lambda=\mathbf{1}_m$, so $\lambda$ is uniquely determined by
$A$, and there are only finitely many choices for $A$.

For $\lambda\in\Lambda_m$ and $m\leq k\leq |U(\lambda)|$, set
\[
C_k(\lambda)
:=
\bigl|\{U'\subseteq U(\lambda): |U'|=k,
\ \operatorname{rank}U'=m\}\bigr|.
\]

\begin{lemma}\label{ConstructiveLemma}
For all $1\leq m\leq n$,
\[
|\mathcal{I}_{n,m}^{>0}|
=
\sum_{\lambda\in\Lambda_m}
\sum_{k=m}^{|U(\lambda)|}
C_k(\lambda)
\sum_{\ell=0}^{k}(-1)^{k-\ell}\binom{k}{\ell}\ell^n.
\]
\end{lemma}

\begin{proof}
For $\lambda\in\Lambda_m$, let
$\mathcal{A}(\lambda):=\{A\in\mathcal{I}_{n,m}^{>0}:\lambda(A)=\lambda\}$.
The uniqueness of the weight vector gives the disjoint union
\(
\mathcal{I}_{n,m}^{>0}
=\bigsqcup_{\lambda\in\Lambda_m}\mathcal{A}(\lambda).
\)
Fix $A\in\mathcal{A}(\lambda)$ and let $U'$ be the set of its distinct rows.
Then $U'\subseteq U(\lambda)$, and $\operatorname{rank}U'=m$ because $A$ has
rank $m$.  Conversely, after choosing such a set $U'$ of cardinality $k$, a
matrix with $U'$ as its set of distinct rows is the same as a surjection from
$[n]$ onto $U'$.  By inclusion--exclusion, the number of these surjections is
\[
\sum_{\ell=0}^{k}(-1)^{k-\ell}\binom{k}{\ell}\ell^n.
\]
Summing first over $U'$, then over $k$ and $\lambda$, proves the formula.
\end{proof}

Combining Lemmas~\ref{ConstructiveLemma} and~\ref{L1} yields the explicit
formula
\begin{equation}\label{eq:constructive-Bnm}
B_{n,m}
=
\frac{1}{m!}
\sum_{\lambda\in\Lambda_m}
\sum_{k=m}^{|U(\lambda)|}
C_k(\lambda)
\sum_{\ell=0}^{k}(-1)^{k-\ell}\binom{k}{\ell}\ell^n.
\end{equation}
Thus, knowing $\Lambda_m$ and the integers $C_k(\lambda)$ determines
$B_{n,m}$ for every $n\geq m$.

The numbers $C_k(\lambda)$ can be computed directly.  The sets $\Lambda_m$
can be generated recursively by the following weight-vector reformulation of
Peleg's induction~\cite{Peleg}; we follow the complete algorithmic formulation
and proof given in~\cite[Section~4.1, especially Theorem~4.5]{Mermoud}.  The
first case of that formulation leaves the weight vector unchanged, while its
remaining three cases give the operations below.  Start with
$\Lambda_1=\{(1)\}$ and
assume that $\Lambda_j$ is known for every $j<m$.  In the third operation,
take all possible zero-padding embeddings.

\begin{enumerate}
    \item Let $\lambda\in\Lambda_{m-1}$ and $I\subseteq[m-1]$ satisfy
    $\sum_{i\in I}\lambda_i<1$.  Insert
    $1-\sum_{i\in I}\lambda_i$ in any coordinate position.

    \item Let $\lambda\in\Lambda_{m-1}$, let $I\subseteq[m-1]$, and let
    $\sigma\in[m-1]\setminus I$ satisfy
    $1-\lambda_\sigma<\sum_{i\in I}\lambda_i<1$.
    Insert $1-\sum_{i\in I}\lambda_i$ in any coordinate position and
    subtract the same quantity from $\lambda_\sigma$.

    \item Let $\lambda\in\Lambda_k$ and $\omega\in\Lambda_\ell$, where
    $k,\ell<m$.  Pad them with zeros to obtain
    $\widehat\lambda,\widehat\omega\in\mathbb{Q}^m$, with no coordinate at
    which both vectors vanish.  For $I\subseteq[m]$, put
    \[
    t_I=
    \frac{1-\sum_{i\in I}\widehat\lambda_i}
    {\sum_{i\in I}\widehat\omega_i-
     \sum_{i\in I}\widehat\lambda_i},
    \]
    whenever the denominator is nonzero.  If $t_I\in(0,1)$, form
    $\lambda'=t_I\widehat\omega+(1-t_I)\widehat\lambda$.
    Retain $\lambda'$ precisely when $\operatorname{rank}U(\lambda')=m$.
\end{enumerate}

For completeness, we trace Peleg's construction to the first step at which the
resulting collection has $m$ members.  Immediately before that step it has
fewer than $m$ members.  Since the first case preserves both the number of
members and the weight vector, the first appearance of $m$ members must occur
in one of the three cases above.  Thus, iterating these operations and
removing duplicates produces $\Lambda_m$.
\begin{samepage}
Substitution into~\eqref{eq:constructive-Bnm} gives, for $m\leq4$,
$B_{n,1}=1$, $B_{n,2}=2^{n-1}-1$, and
\begin{align*}
B_{n,3}&=3^{n-1}-2^n+1,\\
B_{n,4}&=\frac{1}{24}6^n+\frac{7}{24}4^n-2\cdot3^n
          +\frac{29}{8}2^n-\frac{8}{3}.
\end{align*}
\end{samepage}

The values obtained for $n\leq6$ are shown in Table~\ref{Table2}.

\begin{table}[!htbp]
\centering
\caption{Values of $B_{n,m}$.}\label{Table2}
\begin{tabular}{@{}c*{6}{r}@{}}
\toprule
$m\backslash n$ & 1 & 2 & 3 & 4 & 5 & 6 \\
\midrule
1 & 1 & 1 & 1 & 1 & 1 & 1 \\
2 & 0 & 1 & 3 & 7 & 15 & 31 \\
3 & 0 & 0 & 2 & 12 & 50 & 180 \\
4 & 0 & 0 & 0 & 22 & 250 & 1910 \\
5 & 0 & 0 & 0 & 0 & 976 & 18780 \\
6 & 0 & 0 & 0 & 0 & 0 & 179312 \\
\midrule
Total & 1 & 2 & 6 & 42 & 1292 & 200214 \\
\bottomrule
\end{tabular}
\end{table}

The exact formula also determines the complete first-order asymptotics for
every fixed $m$.

\begin{corollary}
For each fixed $m\geq2$, as $n\to\infty$,
\[
B_{n,m}=
\begin{cases}
\displaystyle
\frac{1}{m!}\binom{m}{m/2}^{n}
+O_m\!\left(\left[\binom{m}{m/2}-1\right]^n\right),
& m \text{ even},\\[3mm]
\displaystyle
\frac{2}{m!}\binom{m}{(m+1)/2}^{n}
+O_m\!\left(\left[\binom{m}{(m+1)/2}-1\right]^n\right),
& m \text{ odd}.
\end{cases}
\]
For $m=1$, one has $B_{n,1}=1$.
\end{corollary}

\begin{proof}
By Lemma~\ref{maxU},
$|U(\lambda)|\leq\binom{m}{\lceil m/2\rceil}$ for every
$\lambda\in\Lambda_m$.  The equality cases of Sperner's theorem show that an
antichain attaining this bound is one of the complete middle layers of the
Boolean lattice.  Hence equality forces
$\lambda=2\mathbf{1}_m/m$ if $m$ is even; if $m$ is odd, it forces one of
the two possibilities
\[
\lambda=\frac{2}{m-1}\mathbf{1}_m
\qquad\text{or}\qquad
\lambda=\frac{2}{m+1}\mathbf{1}_m.
\]
Indeed, if $U(\lambda)$ is the complete layer of $r$-subsets, comparison of
two such subsets differing in one element shows that all coordinates of
$\lambda$ are equal, and their value is $1/r$.

For fixed $k$,
\[
\sum_{\ell=0}^{k}(-1)^{k-\ell}\binom{k}{\ell}\ell^n
=k^n+O_k((k-1)^n).
\]
For each weight vector displayed above, the only subset of $U(\lambda)$ of
cardinality $|U(\lambda)|$ is $U(\lambda)$ itself, and this set has rank
$m$.  Thus $C_{|U(\lambda)|}(\lambda)=1$.  Since $\Lambda_m$ is finite,
every remaining term is
$O_m\!\left(\left[\binom{m}{\lceil m/2\rceil}-1\right]^n\right)$.
Substitution into~\eqref{eq:constructive-Bnm} proves the result.
\end{proof}

\section{Direct bounds for finite \texorpdfstring{$n$}{n}}
\label{app:finite-bounds}

This appendix gives bounds that hold for each finite $n$ and contain no
asymptotic notation.  Their proofs are entirely deterministic and do not use
Tikhomirov's theorem.

\begin{proposition}[Finite bounds]
For every $n\geq2$,
$\frac{2^{n^2-n-1}}{n!(2^n-n)}\leq B_n<\frac{4}{n!}\,2^{n^2-n}$.
For $n=1$, one has $B_1=1$.
\end{proposition}

\begin{proof}
We begin with the lower bound.  Regard $\mathbf{1}_n$ as a vector in
$\mathbb{F}_2^n$.  The number of ordered bases over $\mathbb{F}_2$ of the form
$(\mathbf{1}_n,a_2,\ldots,a_n)$
is
\begin{equation}\label{eq:fixed-column-bases}
\prod_{k=1}^{n-1}(2^n-2^k)
=2^{n^2-n}\prod_{j=1}^{n-1}(1-2^{-j}).
\end{equation}
Indeed, after $k$ independent vectors have been chosen, their span contains
exactly $2^k$ vectors.

From such a basis define, in $\mathbb{F}_2^n$,
\[
b_i=a_{i+1}\quad(1\leq i<n),
\qquad
b_n=\mathbf{1}_n+\sum_{i=1}^{n-1}b_i,
\]
and let $M=(b_1,\ldots,b_n)$.  Consider the original basis consisting of
$\mathbf{1}_n,b_1,\ldots,b_{n-1}$.  In this basis the vectors
$b_1,\ldots,b_{n-1}$ are the last $n-1$ coordinate vectors, whereas $b_n$ has
first coordinate $1$.
Thus $b_1,\ldots,b_n$ are linearly independent, so the columns of $M$ form a
basis over $\mathbb{F}_2$.  In particular, $\det M$ is odd, so
$M$ is also nonsingular over $\mathbb{Q}$.  Since
$\mathbf{1}_n=\sum_{i=1}^n b_i$ in $\mathbb{F}_2^n$, replacing any $b_j$
by $\mathbf{1}_n$ again gives a basis.  Thus, for every $j\in[n]$, the system
$\{\mathbf{1}_n\}\cup\{b_i:i\neq j\}$
is a basis over $\mathbb{F}_2$.  If the $j$th coordinate of the unique
solution to $M\lambda=\mathbf{1}_n$ were zero, replacing the $j$th column of
$M$ by $\mathbf{1}_n$ would give a singular matrix over $\mathbb{Q}$.  Its
integer determinant would therefore also vanish modulo~$2$, contradicting
the fact that the displayed system is a basis over $\mathbb{F}_2$.
Thus $M\in\mathcal{I}_{n,n}^{\neq0}$.  The construction is injective because
the first $n-1$ columns of $M$ recover $a_2,\ldots,a_n$.  It follows
from~\eqref{eq:fixed-column-bases} that
\[
|\mathcal{I}_{n,n}^{\neq0}|
\geq2^{n^2-n}\prod_{j=1}^{n-1}(1-2^{-j}).
\]
For numbers $x_j\in[0,1]$, one has
\(\prod_j(1-x_j)\geq1-\sum_jx_j\).  Therefore
\[
\prod_{j=1}^{n-1}(1-2^{-j})
\geq
\frac12\left(1-\sum_{j=2}^{\infty}2^{-j}\right)
=\frac14.
\]
Since $B_n\geq B_{n,n}$, Lemmas~\ref{L1} and~\ref{OrbLemma} give
\[
B_n\geq B_{n,n}
\geq\frac{2|\mathcal{I}_{n,n}^{\neq0}|}{n!(2^n-n)}
\geq\frac{2^{n^2-n-1}}{n!(2^n-n)},
\]
which proves the lower bound.

For the upper bound, first note that $B_{n,1}=1$.  Fix $2\leq m<n$.
The span of the $m$ independent columns of a matrix in
$\mathcal{I}_{n,m}^{\neq0}$ contains at most $2^m$ \zeroone\ vectors.  To see
this, restrict the vectors in the span to the coordinates of any nonsingular
$m\times m$ minor; this restriction is injective.  Hence there are at least
$2^n-2^m$ choices of a \zeroone\ column whose addition preserves full rank.
Appending any such column changes the weight vector $\lambda$ to
$(\lambda,0)$.

For $t\in[m+1]$, let
\[
\mathcal{J}_t:=
\{M'\in\mathcal{I}_{n,m+1}:
\lambda(M')_t=0\text{ and }\lambda(M')_i\neq0
\text{ for every }i\neq t\}.
\]
The sets $\mathcal{J}_t$ are disjoint and have the same cardinality by column
symmetry.
Every extension above belongs to $\mathcal{J}_{m+1}$, and distinct pairs
consisting of a matrix and an appended column give distinct extensions.
Since the union of the $\mathcal{J}_t$ contains at most $2^{n(m+1)}$
\zeroone\ matrices, column symmetry gives
\begin{align*}
|\mathcal{I}_{n,m}^{\neq0}|(2^n-2^m)
&\leq\frac{2^{n(m+1)}}{m+1},\\
|\mathcal{I}_{n,m}^{\neq0}|
&\leq
\frac{2^{nm}}{(m+1)(1-2^{m-n})}.
\end{align*}
By Lemmas~\ref{L1} and~\ref{OrbLemma},
\[
B_{n,m}
\leq
\frac{2^{nm+1}}
{(m+1)!\left(2^m-\binom{m}{\lceil m/2\rceil}\right)
(1-2^{m-n})}
\qquad(2\leq m<n).
\]
Since
\(\binom{m}{\lceil m/2\rceil}\leq2^{m-1}\) and
\(1-2^{m-n}\geq1/2\), this implies
\[
B_{n,m}\leq\frac{8\,2^{m(n-1)}}{(m+1)!}
\qquad(2\leq m<n).
\]
For the square term, the trivial estimate
$|\mathcal{I}_{n,n}^{\neq0}|\leq2^{n^2}$ yields
\[
B_{n,n}\leq
\frac{2^{n^2+1}}
{n!\left(2^n-\binom{n}{\lceil n/2\rceil}\right)}.
\]

Suppose $n\geq5$.  Since
\(\binom{n}{\lceil n/2\rceil}/2^n\leq5/16\), we have
$n!B_{n,n}/2^{n^2-n+1}\leq16/11$.
Setting $k=n-m$ and using $n!/(m+1)!\leq n^{k-1}$, we also obtain
\[
\frac{n!}{2^{n^2-n+1}}
\sum_{m=2}^{n-1}B_{n,m}
\leq
2^{3-n}\sum_{k=1}^{n-2}
\left(\frac{n}{2^{n-1}}\right)^{k-1}
\leq\frac{4}{11}.
\]
Here we used $n/2^{n-1}\leq5/16$ and $2^{3-n}\leq1/4$.
Finally,
\(\frac{n!B_{n,1}}{2^{n^2-n+1}}<2/11\), because
\(2^{n^2-n+1}/n!>11/2\).
Adding the three estimates proves the upper bound for $n\geq5$.
The remaining cases follow from Table~\ref{Table2}:
$B_1=1$, $B_2=2$, $B_3=6$, and $B_4=42$.
\end{proof}

\section{Fractional perfect matchings}
\label{app:fractional-matchings}

Let $H=([n],E)$ be a simple hypergraph all of whose edges are nonempty, and
let $A_H\in\{0,1\}^{[n]\times E}$ be its incidence matrix, where
$(A_H)_{i,e}=1$ if $i\in e$ and $(A_H)_{i,e}=0$ otherwise.  The fractional
perfect-matching polytope of $H$ is
\[
P_{\mathrm{fpm}}(H)
:=\{x\in\mathbb{R}_{\geq0}^{E}:A_Hx=\mathbf{1}\}.
\]
Its points are the fractional perfect matchings of $H$.  Write
$\operatorname{vert}P$ for the vertex set of a polytope $P$.

\begin{lemma}\label{lem:fpm-bound}
For every simple hypergraph $H=([n],E)$ all of whose edges are nonempty,
$\bigl|\operatorname{vert}P_{\mathrm{fpm}}(H)\bigr|\leq B_n$.  As a
consequence,
\mbox{$\bigl|\operatorname{vert}P_{\mathrm{fpm}}(H)\bigr|
=O\!\left(2^{n^2-n}/n!\right)$}
as $n\to\infty$, uniformly in $H$.
\end{lemma}

\begin{proof}
For a subfamily $\Phi\subseteq E$, let
$A_\Phi$ be the submatrix of $A_H$ formed by the columns indexed by $\Phi$.
A point $x\in P_{\mathrm{fpm}}(H)$ is a vertex if and only if
$A_{\operatorname{supp}(x)}$ has full column rank, where
$\operatorname{supp}(x):=\{e\in E:x_e>0\}$.

For completeness, we give a short proof of this criterion.  If the columns
indexed by $\operatorname{supp}(x)$ are linearly dependent, choose a nonzero
vector $d$ supported on $\operatorname{supp}(x)$ with $A_Hd=0$.  For all
sufficiently small $\varepsilon>0$, both $x+\varepsilon d$ and
$x-\varepsilon d$ are nonnegative and feasible, so $x$ is not a vertex.
Conversely, if $x=(y+z)/2$ for feasible $y,z$, then nonnegativity implies
$y_e=z_e=0$ whenever $x_e=0$.  Hence
\[
A_{\operatorname{supp}(x)}
\bigl(y_{\operatorname{supp}(x)}-z_{\operatorname{supp}(x)}\bigr)=0.
\]
If $A_{\operatorname{supp}(x)}$ has full column rank, then $y=z=x$, so $x$
is a vertex.

Let $x$ be a vertex and set $\Phi=\operatorname{supp}(x)$ and
$m=|\Phi|$.  Then $A_\Phi$ has full column rank and
$A_\Phi x_\Phi=\mathbf{1}$ with $x_\Phi>0$.  Therefore
$A_\Phi\in\mathcal{I}_{n,m}^{>0}$, so
Proposition~\ref{prop:matrix-correspondence} implies that $\Phi$ is a minimal
balanced collection.

Conversely, let $\Phi\subseteq E$ be a minimal balanced collection.  By
Proposition~\ref{prop:matrix-correspondence}, $A_\Phi$ has full column rank
and its unique balancing weights are positive.  Extending these weights by zero
on $E\setminus\Phi$ gives a point of $P_{\mathrm{fpm}}(H)$ whose support is
$\Phi$.  Since $A_\Phi$ has full column rank, this point is a vertex.  These two
constructions are inverse to one another, which proves the claimed bijection
and the bound.
\end{proof}

\section{Circuits of the resonance arrangement}
\label{app:resonance-circuits}

Here we describe the connection between minimal balanced collections and
circuits of the resonance arrangement.  For every nonempty $S\subseteq[n]$,
let \( H_S:=\left\{x\in\mathbb{R}^n:\sum_{j\in S}x_j=0\right\}. \)
The \emph{resonance arrangement} is
$\mathcal{A}_n:=\{H_S:\varnothing\neq S\subseteq[n]\}$.  We associate
$H_S$ with its normal vector $\mathbf{1}_S$; these normals are precisely the
nonzero vertices of the Boolean cube $\{0,1\}^n$.  A \emph{circuit} is an
inclusion-minimal linearly dependent collection of these normal vectors.  Let
$\kappa_{n,m}$ denote the number of $m$-element circuits of $\mathcal{A}_n$,
and let $\kappa_n$ denote the total number of its circuits.

Billera et al.~\cite{Billera2} observed that the chambers of
$\mathcal{A}_n$ are in bijection with maximal unbalanced collections on
$[n+1]$.  Laplace Mermoud et al.~\cite{Mermoud} also discuss how balanced
and unbalanced collections relate to the resonance arrangement.
K\"uhne~\cite{Kuhne} studied the circuits and the
broken-circuit complex of $\mathcal{A}_n$ in order to compute its Betti
numbers.  This motivates the enumeration of circuits of each fixed cardinality.

\begin{proposition}\label{prop:resonance-circuits}
For $n\geq1$ and $3\leq m\leq n+1$, one has
$B_{n+1,m}=2\kappa_{n,m}$.  Consequently, $B_{n+1}=2^n+2\kappa_n$.
\end{proposition}

\begin{proof}
Let $\{\mathbf{1}_{T_1},\ldots,\mathbf{1}_{T_m}\}$ be a circuit of
$\mathcal{A}_n$.  Its linear dependence is unique up to multiplication by a
nonzero scalar, so write it as \( \sum_{i=1}^m a_i\mathbf{1}_{T_i}=0. \)
Every $a_i$ is nonzero, and both signs occur.  Define subsets of $[n+1]$ by
\[
S_i:=
\begin{cases}
T_i, & a_i>0,\\
([n]\setminus T_i)\cup\{n+1\}, & a_i<0.
\end{cases}
\]
If $c:=\sum_{a_i<0}|a_i|$, then a direct calculation gives
\(
\sum_{i=1}^m |a_i|\mathbf{1}_{S_i}
=c\mathbf{1}_{n+1}.
\)
Thus $|a_i|/c$ are balancing weights for $\{S_1,\ldots,S_m\}$; minimality
follows from that of the circuit.  Replacing every $a_i$ by $-a_i$ gives the
complementary minimal balanced collection
$\{[n+1]\setminus S_1,\ldots,[n+1]\setminus S_m\}$.

Conversely, let $\Phi=\{S_1,\ldots,S_m\}$ be a minimal balanced collection on
$[n+1]$, where $m\geq3$, and let
\(\sum_i\lambda_i\mathbf{1}_{S_i}=\mathbf{1}_{n+1}\) with every
$\lambda_i>0$.  Set
\[
T_i:=
\begin{cases}
S_i, & n+1\notin S_i,\\
[n+1]\setminus S_i, & n+1\in S_i.
\end{cases}
\]
The $T_i$ are distinct nonempty subsets of $[n]$.  Subtracting the equation
for the last coordinate from those for the first $n$ coordinates gives
\[
\sum_{n+1\notin S_i}\lambda_i\mathbf{1}_{T_i}
-\sum_{n+1\in S_i}\lambda_i\mathbf{1}_{T_i}=0.
\]
The full-rank assertion in Proposition~\ref{prop:matrix-correspondence} shows
that this dependence is minimal.  Since a circuit dependence has exactly two
opposite sign patterns, these constructions prove
$B_{n+1,m}=2\kappa_{n,m}$.  Finally, there are no one- or two-element circuits,
whereas $B_{n+1,1}=1$ and $B_{n+1,2}=2^n-1$.  Summing over $m$ gives
$B_{n+1}=2^n+2\kappa_n$.
\end{proof}

Appendix~A therefore immediately gives
\[
\kappa_{n,3}=\frac{3^n-2^{n+1}+1}{2},
\qquad
\kappa_{n,4}
=\frac18 6^n+\frac7{12}4^n-3^{n+1}
 +\frac{29}{8}2^n-\frac43.
\]
The proposition similarly translates the fixed-$m$ asymptotics of
Appendix~A.  The main result immediately yields the following corollary.

\begin{corollary}\label{cor:resonance-circuits}
As $n\to\infty$,
\[
\kappa_n\sim \kappa_{n,n+1}
\sim\frac{2^{n(n+1)}}{(n+1)!}
\sim\binom{2^n-1}{n+1}.
\]
In particular, a uniformly random $(n+1)$-element subcollection of the normal
vectors of $\mathcal{A}_n$ is a circuit with probability tending to $1$.
\end{corollary}

\section*{Declarations}

\begin{itemize}
    \item This research is supported by the ``Priority 2030'' strategic academic leadership program.
    \item The authors declare that they have no competing interests.
    \item No data are associated with this manuscript.
\end{itemize}

\section*{Acknowledgements}

This work was carried out during the MIPT--LIPS-25 Summer Research Program at
the Phystech School of Applied Mathematics and Computer Science.

The authors are grateful to Andrey Kupavskii and Dylan Laplace Mermoud for
valuable discussions and remarks.  They also thank Dmitriy Gribanov and Roman
Karasev for reviewing the first draft of the paper.

\section*{Declaration of generative AI and AI-assisted technologies in the
manuscript preparation process}

During the preparation and revision of this work, the authors made extensive
use of the GPT-5.6 Sol model to improve the exposition and clarity of the
manuscript and to assist with literature searches and consistency
improvements.  The GPT-5.6 Sol model suggested including Appendix~D and, in
particular, formulating Proposition~D.1 and Corollary~D.1.  The main proofs and overall mathematical ideas were developed by the authors.
The authors verified and edited AI-assisted output and take full responsibility
for the article.

\bibliographystyle{cas-model2-names}
\bibliography{references}

\begin{thebibliography}{14}
\expandafter\ifx\csname natexlab\endcsname\relax\def\natexlab#1{#1}\fi
\providecommand{\url}[1]{\texttt{#1}}
\providecommand{\href}[2]{#2}
\providecommand{\path}[1]{#1}
\providecommand{\DOIprefix}{doi:}
\providecommand{\ArXivprefix}{arXiv:}
\providecommand{\URLprefix}{URL: }
\providecommand{\Pubmedprefix}{pmid:}
\providecommand{\doi}[1]{\href{http://dx.doi.org/#1}{\path{#1}}}
\providecommand{\Pubmed}[1]{\href{pmid:#1}{\path{#1}}}
\providecommand{\bibinfo}[2]{#2}
\ifx\xfnm\relax \def\xfnm[#1]{\unskip,\space#1}\fi
\bibitem[{B{\'a}r{\'a}ny(1982)}]{Barany}
\bibinfo{author}{B{\'a}r{\'a}ny, I.}, \bibinfo{year}{1982}.
\newblock \bibinfo{title}{A generalization of {Carath{\'e}odory's} theorem}.
\newblock \bibinfo{journal}{Discrete Math.} \bibinfo{volume}{40},
  \bibinfo{pages}{141--152}.
\newblock \DOIprefix\doi{10.1016/0012-365X(82)90115-7}.
\bibitem[{Billera et~al.(2012)Billera, Moore, Dufort~Moraites, Wang and
  Williams}]{Billera2}
\bibinfo{author}{Billera, L.J.}, \bibinfo{author}{Moore, J.T.},
  \bibinfo{author}{Dufort~Moraites, C.}, \bibinfo{author}{Wang, Y.},
  \bibinfo{author}{Williams, K.}, \bibinfo{year}{2012}.
\newblock \bibinfo{title}{Maximal unbalanced families}.
\newblock \bibinfo{howpublished}{arXiv preprint}.
\newblock \DOIprefix\doi{10.48550/arXiv.1209.2309},
  \href{http://arxiv.org/abs/1209.2309}{\tt arXiv:1209.2309}.
\bibitem[{Bludov and Musin(2023)}]{BluMus}
\bibinfo{author}{Bludov, M.V.}, \bibinfo{author}{Musin, O.R.},
  \bibinfo{year}{2023}.
\newblock \bibinfo{title}{Balanced 2-subsets}.
\newblock \bibinfo{journal}{Math. Notes} \bibinfo{volume}{114},
  \bibinfo{pages}{407--411}.
\newblock \DOIprefix\doi{10.1134/S0001434623090122}.
\bibitem[{Bondareva(1963)}]{Bon}
\bibinfo{author}{Bondareva, O.N.}, \bibinfo{year}{1963}.
\newblock \bibinfo{title}{Some applications of linear programming methods to
  the theory of cooperative games}.
\newblock \bibinfo{journal}{Probl. Kibern.} \bibinfo{volume}{10},
  \bibinfo{pages}{119--139}.
\newblock \bibinfo{note}{In Russian}.
\bibitem[{Garc{\'i}a-Segador et~al.(2026)Garc{\'i}a-Segador, Grabisch,
  Laplace~Mermoud and Miranda}]{ClosestBalancedGame}
\bibinfo{author}{Garc{\'i}a-Segador, P.}, \bibinfo{author}{Grabisch, M.},
  \bibinfo{author}{Laplace~Mermoud, D.}, \bibinfo{author}{Miranda, P.},
  \bibinfo{year}{2026}.
\newblock \bibinfo{title}{On the closest balanced game}.
\newblock \bibinfo{howpublished}{arXiv preprint}.
\newblock \DOIprefix\doi{10.48550/arXiv.2601.15318},
  \href{http://arxiv.org/abs/2601.15318}{\tt arXiv:2601.15318}.
\bibitem[{Garc{\'i}a-Segador et~al.(2025)Garc{\'i}a-Segador, Grabisch and
  Miranda}]{BalGame}
\bibinfo{author}{Garc{\'i}a-Segador, P.}, \bibinfo{author}{Grabisch, M.},
  \bibinfo{author}{Miranda, P.}, \bibinfo{year}{2025}.
\newblock \bibinfo{title}{On the set of balanced games}.
\newblock \bibinfo{journal}{Math. Oper. Res.} \bibinfo{volume}{50},
  \bibinfo{pages}{2047--2072}.
\newblock \DOIprefix\doi{10.1287/moor.2023.0379}.
\bibitem[{Howard and Sisson(1996)}]{GeneralPutnam}
\bibinfo{author}{Howard, R.}, \bibinfo{author}{Sisson, P.},
  \bibinfo{year}{1996}.
\newblock \bibinfo{title}{Capturing the origin with random points:
  Generalizations of a {Putnam} problem}.
\newblock \bibinfo{journal}{Coll. Math. J.} \bibinfo{volume}{27},
  \bibinfo{pages}{186--192}.
\newblock \DOIprefix\doi{10.1080/07468342.1996.11973774}.
\bibitem[{Klosinski et~al.(1993)Klosinski, Alexanderson and Larson}]{Putnam}
\bibinfo{author}{Klosinski, L.F.}, \bibinfo{author}{Alexanderson, G.L.},
  \bibinfo{author}{Larson, L.C.}, \bibinfo{year}{1993}.
\newblock \bibinfo{title}{The fifty-third {William Lowell Putnam Mathematical
  Competition}}.
\newblock \bibinfo{journal}{Am. Math. Mon.} \bibinfo{volume}{100},
  \bibinfo{pages}{755--767}.
\newblock \DOIprefix\doi{10.1080/00029890.1993.11990483}.
\bibitem[{K{\"u}hne(2023)}]{Kuhne}
\bibinfo{author}{K{\"u}hne, L.}, \bibinfo{year}{2023}.
\newblock \bibinfo{title}{The universality of the resonance arrangement and its
  {Betti} numbers}.
\newblock \bibinfo{journal}{Combinatorica} \bibinfo{volume}{43},
  \bibinfo{pages}{277--298}.
\newblock \DOIprefix\doi{10.1007/s00493-023-00006-x}.
\bibitem[{Laplace~Mermoud et~al.(2023)Laplace~Mermoud, Grabisch and
  Sudh{\"o}lter}]{Mermoud}
\bibinfo{author}{Laplace~Mermoud, D.}, \bibinfo{author}{Grabisch, M.},
  \bibinfo{author}{Sudh{\"o}lter, P.}, \bibinfo{year}{2023}.
\newblock \bibinfo{title}{Minimal balanced collections and their application to
  core stability and other topics of game theory}.
\newblock \bibinfo{journal}{Discrete Appl. Math.} \bibinfo{volume}{341},
  \bibinfo{pages}{60--81}.
\newblock \DOIprefix\doi{10.1016/j.dam.2023.07.025}.
\bibitem[{Peleg(1965)}]{Peleg}
\bibinfo{author}{Peleg, B.}, \bibinfo{year}{1965}.
\newblock \bibinfo{title}{An inductive method for constructing minimal balanced
  collections of finite sets}.
\newblock \bibinfo{journal}{Nav. Res. Logist. Q.} \bibinfo{volume}{12},
  \bibinfo{pages}{155--162}.
\newblock \DOIprefix\doi{10.1002/nav.3800120203}.
\bibitem[{Sanderson(2017)}]{TheHardestProblem}
\bibinfo{author}{Sanderson, G.}, \bibinfo{year}{2017}.
\newblock \bibinfo{title}{The hardest problem on the hardest test}.
\newblock \bibinfo{howpublished}{3Blue1Brown, YouTube video}.
\newblock \URLprefix \url{https://www.youtube.com/watch?v=OkmNXy7er84}.
  \bibinfo{note}{published 8 December 2017; accessed 28 August 2026}.
\bibitem[{Shapley(1967)}]{Sh67}
\bibinfo{author}{Shapley, L.S.}, \bibinfo{year}{1967}.
\newblock \bibinfo{title}{On balanced sets and cores}.
\newblock \bibinfo{journal}{Nav. Res. Logist. Q.} \bibinfo{volume}{14},
  \bibinfo{pages}{453--460}.
\newblock \DOIprefix\doi{10.1002/nav.3800140404}.
\bibitem[{Tikhomirov(2020)}]{Tikhomirov2020}
\bibinfo{author}{Tikhomirov, K.}, \bibinfo{year}{2020}.
\newblock \bibinfo{title}{Singularity of random {Bernoulli} matrices}.
\newblock \bibinfo{journal}{Ann. Math.} \bibinfo{volume}{191},
  \bibinfo{pages}{593--634}.
\newblock \DOIprefix\doi{10.4007/annals.2020.191.2.6}.

\end{thebibliography}

\end{document}